\documentclass[12pt,reqno]{amsart}

\usepackage{graphicx,indentfirst}
\usepackage{amsmath,amssymb,mathrsfs}
\usepackage{amsthm,amscd}
\usepackage{verbatim}
\usepackage{appendix}
\usepackage{enumitem,titletoc}
\usepackage{appendix}
\usepackage{imakeidx}
\makeindex[columns=2, title=Alphabetical Index]
\usepackage{fancyhdr}
\usepackage{amsfonts,color}
\usepackage[all]{xy}
\usepackage{tikz-cd}
\usepackage{syntonly}
\usepackage{fancyhdr}
\newtheorem{theorem}{Theorem}[section]
\newtheorem{lemma}{Lemma}[section]
\newtheorem{remark}{Remark}[section]
\usepackage[left=2.55cm,right=2.55cm,bottom=2.8cm,top=2.8cm]{geometry}

\newcommand{\abs}[1]{|#1|^2}

\usepackage{tikz}
\usepackage{extarrows}

\def\XXint#1#2#3{{\setbox0=\hbox{$#1{#2#3}{\int}$ }
\vcenter{\hbox{$#2#3$ }}\kern-.6\wd0}}

\newtheorem{corr}{Corollary}[section]

\allowdisplaybreaks

\newcommand{\ddbar}{i\partial\bar\partial}
\newcommand{\innpro}[1]{\langle#1\rangle}
\newcommand{\bk}[1]{\Big(#1\Big)}

\newcommand{\ho}{ (\frac{\partial}{\partial t}- \Delta)  }
\makeindex

\newcommand{\xa}{\xrightarrow}

\DeclareMathOperator{\ric}{Ric}

\DeclareMathOperator{\tr}{tr}

\numberwithin{equation}{section}

\begin{document}

\address{Department of Mathematics, Columbia University, New York, NY 10027}

\email{bguo@math.columbia.edu}


\title{K\"ahler-Ricci flow on blowups along  submanifolds}
\author{ Bin Guo}

\thanks{}
\maketitle
\begin{abstract}
In this short note, we study the behavior of K\"aher-Ricci flow on K\"ahler manifolds which contract divisors to smooth submanifolds. We show that the K\"ahler potentials are H\"older continuous and the flow converges sequentially in Gromov-Hausdorff topology to a compact metric space which is homeomorphic to the base manifold. 

\end{abstract}

\section{Introduction}
The Ricci flow, introduced by Hamilton (\cite{H}) in 1982, has been a powerful tool in solving problems in geometry and analysis. It deforms any metric with positive Ricci curvature in real $3$-dimensional manifold to a metric with constant curvature (\cite{H}). By performing surgery through singular times, Perelman (\cite{P}) used Ricci flow to solve the geometrization conjecture for $3$-dimensional manifolds. On the complex aspect, the Ricci flow preserves the K\"ahler condition (\cite{C}) and is reduced to a scalar equation with Monge-Amp\`ere type, which after suitable normalization converges to a solution of the Calabi conjecture (\cite{Y, C}). The non-K\"ahler analogue of Ricci flow also generates much interest recently, among them are the Chern-Ricci flow (\cite{TW1}),  the Anomaly flow (\cite{PPZ1}) and etc, and we refer to \cite{PPZ2} for a survey on the recent development of non-K\"ahler geometric flows.

The analytic minimal model program, laid out in \cite{ST2}, predicts how the K\"ahler-Ricci flow behaves on a projective variety. It is conjectured that the K\"ahler-Ricci flow will either collapse in finite time, or deform any projective variety to its minimal model after finitely many divisorial contractions or flips in the Gromov-Hausdorff topology. There are various results on the finite time collapsing of K\"ahler-Ricci flow, see for example \cite{S1,SSW, PSSW, TWY, ToZh} and references therein. The behavior of K\"ahler-Ricci flow on some small contractions is studied in \cite{S13, SY} and it is shown that the flow forms a continuous path in Gromov-Hausdorff topology. In \cite{SW1,SW2}, Song and Weinkove study the divisorial contractions when the divisor is contracted to discrete  points, and it is shown that the flow converges in Gromov-Hausdorff topology to a metric space which is isometric to the metric completion of the base manifold with the smooth limit of the flow outside the divisor, and the flow can be continued on the new space.  The main purpose of this note is to generalize their results to divisorial contractions when the divisor is contracted to a higher dimensional subvariety.

\smallskip

Let $Y$ be a K\"ahler manifold and $N\subset Y$ be a complex submanifold of codimension $k\ge 1$. Let $X$ be the K\"ahler manifold obtained by blowing up $Y$ along $N$, $\pi: X\to Y$ be the blown-down map and $E= \pi^{-1}(N)$ be the exceptional divisor in $X$. We consider the  (unnormalized) K\"ahler-Ricci flow on $X$:
\begin{equation}\label{eqn:KRF}
\left\{\begin{aligned}
&\frac{\partial \omega}{\partial t} = - \ric(\omega),\\
&\omega(0) = \omega_0,
\end{aligned}\right.
\end{equation}
for a suitable fixed K\"ahler metric $\omega_0$ on $X$. We  assume the limit cohomology class satisfies $[\omega_0] + T K_X = [\pi^* \omega_Y]$ for some K\"ahler metric $\omega_Y$ on $Y$, where the maximal existence time (see \cite{TZ}) of the flow  \eqref{eqn:KRF} is given by $$T = \sup\{t>0: ~ [\omega_0] + t K_X \text{ is K\"ahler}\}<\infty.$$
 We define the reference metrics along the flow $$\hat \omega_t = \frac{T-t}{T} \omega_0 + \frac{t}{T}\pi^*\omega_Y.$$ In the following for notation simplicity we shall denote $\hat  \omega_Y = \pi^*\omega_Y$, which is a nonnegative $(1,1)$-form on $X$.

It is well-known that the flow  \eqref{eqn:KRF} is equivalent to the following parabolic complex Monge-Amp\`ere equation
\begin{equation}\label{eqn:MA}
\left\{\begin{aligned}
&\frac{\partial\varphi}{\partial t} = \log \frac{(\hat \omega_t + \ddbar \varphi)^n}{\Omega},\\
&\varphi( 0 ) = 0,
\end{aligned}\right.
\end{equation}
where $\omega = \hat \omega_t + \ddbar \varphi$ satisfies \eqref{eqn:KRF} and $\Omega$ is a smooth volume form satisfying $\ddbar \log \Omega = \frac{1}{T}( \hat\omega_Y - \omega_0  )$. 

Our main theorem is on the behavior of the metrics $\omega(t)$ as $t\to T^-$. 
\begin{theorem}\label{thm:main}
Let $\pi: X\to Y$ and $\omega_t = \omega_0 + \ddbar \varphi_t$ be  as above, then the following hold: there exists a uniform constant $C=C(n,\omega_0,\omega_Y, \pi)>0$
\begin{enumerate}[label=(\arabic*)]
\item $\varphi_t$ is uniformly H\"older continuous in $(X,\omega_0)$, i.e. $|\varphi_t(p) - \varphi_t(q)|\le C d_{\omega_0}(p,q)^\delta$, for any $p,q\in X$ and some $\delta\in (0,1)$, and $\varphi_t\xrightarrow{C^\delta(X,\omega_0)} \varphi_T\in PSH(X,\pi^*\omega_Y)\cap C^\delta(X,\omega_0)$. Moreover, $\varphi_T$ descends to a function $\bar \varphi_T\in PSH(Y,\omega_Y)\cap C^{\delta_0}(Y,\omega_Y)$ for some $\delta_0\in (0,1)$.

\smallskip

\item \label{item 2} $\omega_t$ converge weakly to $\omega_T: = \pi^*\omega_Y + \ddbar \varphi_T$ as $(1,1)$-currents on $X$ and the convergence is smooth and uniform on any compact subset $K\Subset X\backslash E$.

\smallskip

\item diam$(X,\omega_t)\le C$ for any $t\in [0,T)$.

\smallskip

\item \label{item 3}for any sequence $t_i\to T^-$, there exists a subsequence $\{t_{i_j}\}$ such that $(X,\omega_{t_{i_j}})$ (as compact metric spaces) converge in Gromov-Hausdorff topology to a compact metric space $(Z,d_Z)$.

\smallskip

\item the metric completion of $(Y\backslash N, \omega_T)$ is isometric to $(Y,d_T)$, where the distance function $d_T$ is induced from $\omega_T$ and defined in \eqref{eqn:dT}. And there exists an open dense subset $Z^\circ \subset Z$ such that $(Y\backslash N, d_T)$ and $(Z^\circ, d_Z)$ are homeomorphic and locally isometric. Furthermore $(Z,d_Z)$ is homeomorphic to $(Y,d_T)$.

\end{enumerate}

\end{theorem}

The item \ref{item 2} is known to hold for K\"ahler-Ricci flow for more general holomorphic maps $\pi:X\to Y$ with dim$Y = \mathrm{dim} X$ (see e.g. \cite{PSS, SW1, TZ}). We include it in the theorem just for completeness. We remark that Theorem \ref{thm:main} also holds if the base $Y$ has some mild singularities, for example, if the analytic subvariety $N$ is locally of the form $\mathbb C^k\times ( \mathbb C^{n-k}/\mathbb Z_p  )$, where $\mathbb Z_p$ denotes the $S^1$-action $\{e^{2l\pi i/ p}\}_{l=1}^p$ on $\mathbb C^{n-k}$ by $$e^{2l\pi i/p}\cdot (z_{k+1},\ldots, z_n) \to (e^{2l\pi i/p} z_{k+1},\ldots, e^{2l\pi i/p} z_n).$$ The proof is by combining the techniques of \cite{SW2} and this note, so we omit the details.

\smallskip

Lastly we mention that under the same set-up as in Theorem \ref{thm:main}, the same and even stronger results hold for K\"ahler metrics along continuity method. More precisely, let $u_t\in PSH(X, \hat \omega_Y + t \omega_0)$ be the solution to the complex Monge-Amp\`ere equations
\begin{equation}\label{eqn:continuity}\omega_t^n = (\hat \omega_Y + t \omega_0 + \ddbar u_t  )^n = c_te^{F} \omega_0^n,\, \sup u_t = 0, \,  t\in (0,1],\end{equation}
where $F$ is a given smooth function on $X$ and $c_t$ is a normalizing constant so that the integral of both sides are the same. It has been shown in \cite{FGS} that $\mathrm{diam}(X, \omega_t)$ is bounded by a constant $C=C(n,\omega_0, \hat \omega_Y ,F)>0$ and  the Ricci curvature of $\omega_t$ is uniformly bounded below. We can repeat almost identically the proof of Theorem \ref{thm:main} to the equation \eqref{eqn:continuity} to get the same conclusions for $u_t$ as the $\varphi_t$ in Theorem \ref{thm:main}. Furthermore, along the continuity method \eqref{eqn:continuity}, we can improve the Gromov-Hausdorff convergence in Theorem \ref{thm:main} in the sense that the full sequence (without the need of passing to a subsequence) $(X,\omega_t)$ converges in GH topology to a compact metric space $(Z,d_Z)$ which is isometric to the metric completion of $(Y\backslash N, \hat\omega_0)$, where $\hat\omega_0$ is the smooth limit of $\omega_t$ on $X\backslash \pi^{-1}(N) = Y\backslash N$. The main advantage in this case is that  the Ricci curvature has uniform lower bound so we can apply the argument in \cite{DGSW}, in particular the Gromov's lemma to find an almost geodesic connecting any two points away from the singular set $\pi^{-1}(N)$.

\section{Preliminaries}

The following estimates are well-known (\cite{Y, PSS, TZ, SW1}),  so we just state the results and omit  the proofs.
\begin{lemma}\label{lemma 1.1}
There exists a constant $C>0$ depending only on $(X,\omega_0)$, $(Y,\omega_Y)$ such that 
\begin{enumerate}[label=(\roman*)]
\item  $\| \varphi\|_{L^\infty(X)}\le C$ for all $t\in [0,T)$,

\item $\dot\varphi: = \frac{\partial \varphi}{\partial t}\le C$ and this is equivalent to $\omega^n \le C \Omega$ from the equation \eqref{eqn:MA}.

\item as $t\to T^-$, $\varphi$ converge to a bounded $\hat \omega_Y$-PSH function $\varphi_T$ and $\omega$ converge weakly to $\omega_T: = \hat\omega_Y + \ddbar \varphi_T$ as $(1,1)$-currents on $X$.

\end{enumerate}
\end{lemma}

\begin{lemma}\label{lemma 1.2}
There exists a uniform constant $C>0$ such that 
\begin{enumerate}[label=(\roman*)]
\item \label{item lemma 1.2} $\hat \omega_Y \le C \omega$ for all $t\in [0,T)$,

\item for any compact subset $K \Subset X\backslash E$, there exists a constant $C_{j,K}>0$ such that $\| \varphi\|_{C^j(K,\omega_0)}\le C_{j,K}$. Therefore the convergence $\omega_t\to \omega_T$ and $\varphi\to \varphi_T$ is smooth on $X\backslash E$, so $\omega_T$ and $\varphi_T $ are both smooth on $X\backslash E$.

\end{enumerate}
\end{lemma}


In the proof of Lemma \ref{lemma 1.2}, we need the following Chern-Lu inequality as in the proof the Schwarz lemma (\cite{ST0})
\begin{equation*}
(\frac{\partial}{\partial t} - \Delta) \log \tr_\omega \hat \omega_Y \le C \tr_\omega \hat \omega_Y,
\end{equation*}
where $C>0$ depends also on the upper bound of the bisectional curvature of $(Y,\omega_Y)$. In turn this implies the equation below which will be used later.
\begin{equation}\label{eqn:my 1}
(\frac{\partial}{\partial t} - \Delta) \tr_\omega \hat \omega_Y \le - \frac{|\nabla \tr_\omega \hat \omega_Y|^2}{\tr_\omega\hat \omega_Y} + C (\tr_\omega \hat \omega_Y)^2 \le - c_0 |\nabla \tr_\omega\hat \omega_Y|^2 + C,
\end{equation}
where $c_0 = C^{-1}>0$ is the receptacle of the constant $C$ in \ref{item lemma 1.2} Lemma \ref{lemma 1.2}.

\subsection{K\"ahler metrics from the blown up} We will construct a smooth function $\sigma_Y$ on $Y$ such that $\sigma_Y = 0$ precisely on $N$. Choose a finite open cover $\{V_\alpha\}_{\alpha= 1}^J$ of $N$ in $Y$ and complex coordinates $\{w_{\alpha,i}\}_{i=1}^n$ on $V_\alpha$ such that $N\cap V_\alpha = \{w_{\alpha, 1} = \cdots = w_{\alpha, k} = 0\}$. We also denote $V_0 = Y\backslash \cup_\alpha \frac{1}{2}\overline{V_\alpha}$ and we may also assume that $V_0\cap N = \emptyset$. Take a partition of unity $\{\theta_\alpha\}_{\alpha = 0}^J$ subordinate to the open cover $\{V_\alpha\}_{\alpha = 0}^J$, and we define a smooth function
$$\sigma_Y = \theta_0\cdot 1 \;+\; \sum_{\alpha = 1}^J \theta_\alpha \cdot \sum_{j=1}^k |w_{\alpha, j}|^2\in C^\infty (Y),$$
 and it is straightforward to see from the construction that $\sigma_Y$ vanishes precisely along $N$. Since $\{w_{\alpha,i}\}_{i=1}^k$ are defining functions of $N$, it follows that if $V_\alpha \cap V_\beta\neq \emptyset$, then the function $$f_{\alpha\beta} := \frac{\sum_{j=1}^k |w_{\alpha, j}|^2}{\sum_{j=1}^k |w_{\beta,j}|^2},\quad \text{on }V_\alpha\cap V_\beta$$
is never-vanishing and bounded from above. Since the cover is finite we have
\begin{equation}\label{eqn:my 2}
0<c\le \inf_{\alpha,\beta} \inf_{y\in V_\alpha \cap V_\beta\neq \emptyset} f_{\alpha\beta}(y)\le \sup_{\alpha,\beta} \sup_{y\in V_\alpha \cap V_\beta\neq \emptyset} f_{\alpha\beta}(y) \le C<\infty.
\end{equation}
We denote $\sigma_X = \pi^* \sigma_Y$ to be the pulled-back of $\sigma_Y$ to $X$.
\begin{lemma}[see also \cite{PS}]\label{lemma 1.3}
There exists an $\varepsilon_0>0$ such that for all $\varepsilon\in (0, \varepsilon_0]$ the $(1,1)$-form 
$$\omega_\varepsilon: = \pi^*\omega_Y + \varepsilon \ddbar \log \sigma_X$$
is positive definite on $X\backslash E$ and extends to a smooth K\"ahler metric on $X$.
\end{lemma}
\begin{proof}
We only need to prove the positivity of $\omega_\varepsilon$ near $E$, which is in fact local. So we may assume the map $\pi$ is defined from an open set $U\subset X$ to $V_\alpha$ given by $$w_{\alpha, 1} = z_1,\, w_{\alpha,2} = z_1 z_2,\cdots, w_{\alpha, k} = z_1 z_k,\, w_{\alpha, k+1} = z_{k+1},\cdots, w_{\alpha, n} = z_n,$$
where $\{z_i\}$ are the complex coordinates on $U$ such that $E\cap U = \{z_1 = 0\}$. It loses no loss of generality by assuming $\omega_Y$ on $V_\alpha$ is just the Euclidean metric $\omega_{\mathbb C^n} = \sum_j  i dw_{\alpha, j}\wedge d\bar w_{\alpha, j}$.

We note that on $V_\alpha$ $$\sigma_Y =\Big( \sum_{\beta = 1, V_\beta\cap V_\alpha \neq \emptyset}^J \theta_\beta f_{\beta \alpha}\Big) \cdot \sum_{j=1}^k |w_{\alpha, j}|^2 = : \phi_\alpha \cdot \sum_{j=1}^k |w_{\alpha,j}|^2.$$
From \eqref{eqn:my 2}, we know that $\phi_\alpha$ is a smooth function with a strict positive lower bound on $V_\alpha$. In particular $\omega_Y + \varepsilon \ddbar \log \phi_\alpha>0$ on $V_\alpha$ for any $0< \varepsilon\le \varepsilon_0<<1$.  

We calculate
\begin{equation}\label{eqn:pullback}\begin{split}
\pi^* \omega_Y =& (1+\sum_{j=2}^k |z_j|^2) dz_1\wedge d\bar z_1 + \sum_{j=2}^k (z_1 \bar z_j dz_j \wedge d\bar z_1 + \bar z_1 z_j dz_1 \wedge d\bar z_j)\\
& \quad + |z_1|^2 \sum_{j=2}^k dz_j \wedge d\bar z_j +\sum_{j=k+1}^n dz_j\wedge d\bar z_j,
 \end{split}\end{equation}
and note that on $U$ $$\sigma_X = \log \phi_\alpha + \log  |z_1|^2 + \log (1+ \sum_{j=2}^k |z_j|^2  ),$$ so on $U\backslash E$ we have
\begin{equation}\label{eqn:pullback 1}
\ddbar \log \sigma_X = \ddbar \log \phi_\alpha + \frac{\sum_{i, j=2}^k ( (1+|z'|^2)\delta_{ij} - \bar z_i z_j) \sqrt{-1}dz_i \wedge d\bar z_j}{(1+|z'|^2)^2},
\end{equation}
where $z' = (z_2,\ldots, z_k)$ and the second term on RHS is nonnegative in $z'$-directions, which is just the Fubini-Study metric in the coordinates $z'$. By straightforward calculations, we see that if $\varepsilon$ is small enough the $(1,1)$-form $\pi^* \omega_Y + \varepsilon \ddbar \log \sigma_X$ is positive on $X\backslash E$ and extends to a K\"ahler metric on $X$.

\end{proof}
\begin{remark}
Globally from the above calculations we see that $$\omega_\varepsilon = \pi^*\omega_Y + \varepsilon \ddbar \log \sigma_X - \varepsilon [E],$$where $[E]$ denotes the current of integration along $E$.
\end{remark}

We will denote $\omega_X = \pi^* \omega_Y + \varepsilon_0\ddbar \log \sigma_X - \varepsilon_0[E]$ to be a fixed K\"ahler metric obtained from the Lemma \ref{lemma 1.3}. The following inequality follows from the local expression of $\pi^* \omega_Y$ as in the proof of Lemma \ref{lemma 1.3}.
\begin{lemma} 
There exists a uniform constant $C>1$ such that
\begin{equation}\label{eqn:my 3}
C^{-1} \hat \omega_Y \le \omega_X \le  \frac{C}{\sigma_X}\hat \omega_Y,
\end{equation}
where the second inequality is understood on $X\backslash E$.
\end{lemma}

\section{The proof of the main theorem}

Now we are ready to derive the crucial estimates on $\omega$ along the K\"ahler-Ricci flow \eqref{eqn:KRF}.
\begin{lemma}\label{lemma 1.5}
There exists uniform constants $C>0$ and $\delta\in (0,1)$ such that along the flow \eqref{eqn:KRF} we have 
\begin{equation}\label{eqn:main}
\omega\le  C \frac{\omega_0}{\sigma_X^{1-\delta}},\quad \text{on }X\backslash E \times [0, T).
\end{equation}

\end{lemma}
The proof is almost the same as that of Lemma 2.5 in \cite{SW1}, with minor modification using Lemma \ref{lemma 1.3}. For completeness, we provide a sketched proof. 
\begin{proof}
Fix an $\epsilon\in (0,1)$ and define $$Q_\epsilon = \log \tr_{\omega_0} \omega + A \log \sigma_X^{1+\epsilon} \tr_{\hat\omega_Y }\omega - A^2 \varphi,$$
where $A>0$ is a constant to be determined later. First of all, $Q_\epsilon|_{t= 0}\le C$ for a constant $C$ independent of $\epsilon\in (0,1)$, which can be seen from \eqref{eqn:my 3}. Observe that for each time $t_0\in (0,T)$, $\max_X Q_\epsilon$ can only be achieved on $X\backslash E$, since $Q_\epsilon(x)\to -\infty $ as $x\to E$. Thus we assume the maximum of $Q_\epsilon$ is obtained at $(x_0,t_0)$ for some $x_0\in X\backslash E$. From the Chern-Lu inequality the following holds on $X\backslash E$
\begin{align*}
(\frac{\partial}{\partial t} - \Delta) Q_\epsilon \le C \tr_\omega \omega_0 - A \tr_\omega ( A\hat \omega_t + (1+\epsilon)\ddbar \log \sigma_X)+ A^2 \log\frac{\Omega}{\omega^n} + C,
\end{align*} 
where the constant $C$ depends on the lower bound of the bisectional curvature of $(X,\omega_0)$ and the upper bound of bisectional curvature of $(Y,\omega_Y)$. Since $\hat \omega_t\ge c_1 \hat \omega_Y$ for a uniform $c_1>0$ and any $t\in [0,T)$, by Lemma \ref{lemma 1.3} for $A>0$ large enough  $A \hat \omega_t + (1+ \epsilon )\ddbar \log \sigma_X \ge c_2 \omega_0$ on $X\backslash E$ for some $c_2>0$. If $A>0$ is taken even larger then at $(x_0,t_0)$, we have
\begin{align*}
0\le (\frac{\partial}{\partial t} - \Delta) Q_\epsilon \le - 2 \tr_\omega \omega_0 +  A^2 \log\frac{\Omega}{\omega^n} + C\le - \tr_\omega\omega_0 + C,
\end{align*} 
where in the last inequality we use 
$$-\tr_\omega \omega_0 + A^2 \log \frac{\Omega}{\omega^n}\le -\tr_\omega \omega_0 + n A^2 \log \tr_\omega \omega_0 + C \le C,$$
as $\log x \le \varepsilon x + C(\varepsilon)$ for any $x\in (0,\infty)$. So we have $\tr_\omega \omega_0(x_0,t_0)\le C$. Then 
\begin{equation*}
\tr_{\omega_0}\omega|_{(x_0,t_0)}\le \frac{\omega^n}{\omega_0^n } (\tr_\omega \omega_0)^{n-1}|_{(x_0,t_0)}\le C.
\end{equation*}
Observing that from \eqref{eqn:my 3}, $\sigma_X \tr_{\hat \omega_Y} \omega \le C \tr_{\omega_0} \omega$ on $X\backslash E$, thus $\sup_X Q_\epsilon\le C$ for some uniform constant $C>0$. Letting $\epsilon\to 0$, we  get 
\begin{equation*}
\log \tr_{\omega_0} \omega + A \log \sigma_X \tr_{\hat \omega_Y } \omega \le C,\quad \text{on }X\backslash E \times [0,T).
\end{equation*}
Finally from $ C \tr_{\hat \omega_Y}\omega \ge \tr_{\omega_0}\omega$ we see from the above that 
$$\log \tr_{\omega_0} \omega + \log \sigma_X^A (\tr_{\omega_0} \omega)^A \le C,$$
so $\tr_{\omega_0}\omega \le C {\sigma_X^{- A/(1+A)}}$ on $X\backslash E$, and we can then take $\delta = \frac{1}{1+A}\in (0,1)$.
\end{proof}

Next we will show the distance function defined by $ \omega_t$ is H\"older-continuous with respect to the fixed metric $(X,\omega_0)$. 
\begin{lemma}\label{lemma 1.6}
There exists a uniform constant $C>0$ such that for any $p,q\in X$, it holds that
\begin{equation*}
d_{\omega_t}(p,q)\le C d_{\omega_0}(p,q)^\delta,\quad \forall ~ t\in [0,T),
\end{equation*}
where $\delta \in (0,1)$ is the constant determined in Lemma \ref{lemma 1.5}.
\end{lemma}  
\begin{proof}
It suffices to prove the estimate near $E$, say on $T(E)$, a tubular neighborhood of $E$, since $\omega_t$ is uniformly equivalent to $\omega_0$ outside $T(E)$. Choose coordinates charts $\{U_\alpha\}$ covering $T(E)$ and local coordinates $\{z_{\alpha, i}\}_{i=1}^n$ such that $U_\alpha\cap E = \{z_{\alpha,1} = 0\}$. We may assume that the cover is fine enough such that any $p,q\in T(E)$ with $d_{\omega_0}(p,q)\le \frac 1 2 $ must lie in the same $U_\alpha$.  Since we have only finitely many such $U_\alpha$, we will work on one of them only and omit the subscript $\alpha$ for simplicity. Furthermore the fixed K\"ahler metric $\omega_0$ is uniformly equivalent to the Euclidean metric $\omega_{\mathbb C^n}$ on $U$, so without loss of generality we assume $\omega_0 = \omega_{\mathbb C^n}$ on $U$. Recall that Lemma \ref{lemma 1.5} implies that on $U\backslash E$ it holds that
\begin{equation}\label{eqn:my 4}
\omega_t \le C \frac{\omega_{\mathbb C^n}}{|z_1|^{2(1-\delta)}},\quad \forall t\in [0,T),
\end{equation} since $\sigma_X \sim |z_1|^2$ on $U$.

Take any two points $p,q\in U$ with $d_{\omega_0}(p,q) = d< \frac 1 4$. We will consider different cases depending on the positions of $p,q$.

\medskip

\noindent $\bullet$ {\bf Case 1:} $p,q\in E$. Rotate the coordinates if necessary we may assume $p = 0$ and $q = (0,d,0,\ldots,0)$. We pick two points $\tilde p = (d, 0 ,\ldots, 0)$ and $\tilde q = (d,d,0,\ldots, 0)$ as shown the picture below. From \eqref{eqn:my 4}, we have
\begin{figure}[h]
\includegraphics[scale = 0.5]{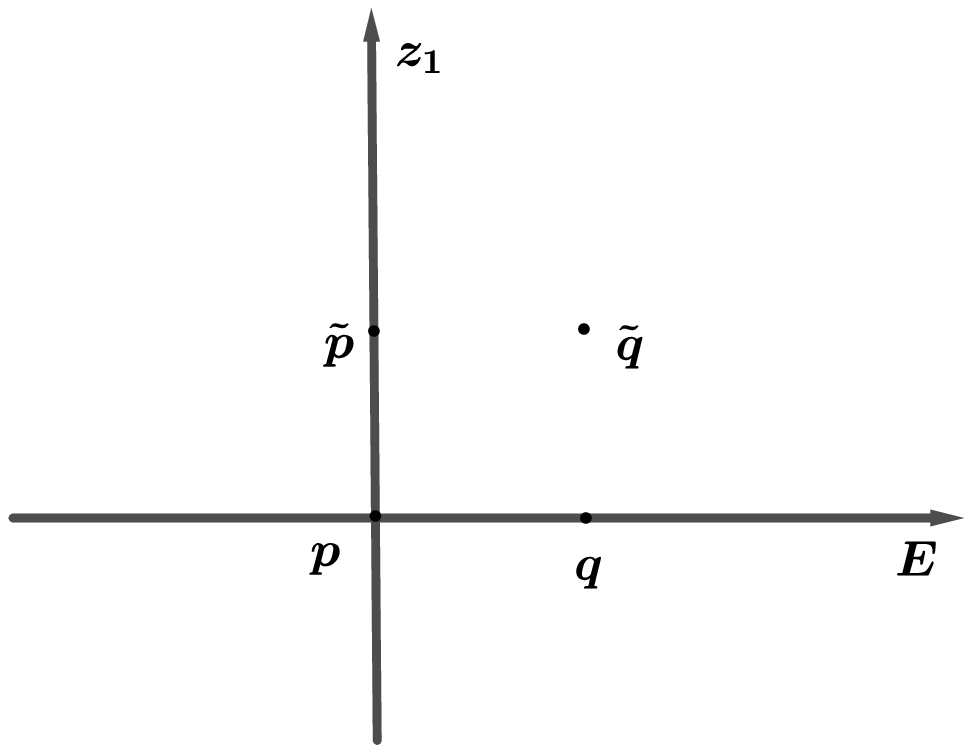}
\end{figure}
\begin{equation*}
d_{\omega_t} (p,\tilde p) \le L_{\omega_t}(\overline{p \tilde p})\le C \int_0^d \frac{1}{r^{1-\delta}}dr \le C d^\delta,
\end{equation*}where $\overline{p\tilde p}$ denotes the (Euclidean) line segment connecting $p$ and $\tilde p$. 
Similarly $d_{\omega_t}(q,\tilde q)\le C d^\delta$. On the other hand, 
$$d_{\omega_t}(\tilde p,\tilde q) \le L_{\omega_t}(\overline{\tilde p\tilde q}) \le \frac{C}{d^{1-\delta}} L_{\omega_{\mathbb C^n}}(\overline{\tilde p\tilde q}) = C d^\delta.$$ If we denote $\gamma = \overline{p\tilde p} + \overline{\tilde p\tilde q} + \overline{q \tilde q}$ to be the piecewise line segment connecting $p$ and $q$, then we have
%
\begin{equation*}
d_{\omega_t}(p,q)\le L_{\omega_t}(\gamma)\le C d^\delta = C d_{\omega_0}(p,q)^{\delta}.
\end{equation*}
We remark that $\gamma\subset X\backslash E$, except the two end points $p,q$.

\medskip

\noindent $\bullet$ {\bf Case 2.} $\min(d_{\omega_0}(p,E), d_{\omega_0}(q,E) )\le d$. The (Euclidean) projections of $p,q$ to $E$, denoted by $p', q'$, respectively, must satisfy $d_{\omega_0}(p',q')\le d$. From the assumption it follows that $d_{\omega_0}(p,p')\le 2d$ and $d_{\omega_0}(q,q')\le 2d$. By similar arguments as above using \eqref{eqn:my 4} we have 
$$d_{\omega_t}(p,p')\le C d^\delta,\quad d_{\omega_t}(q,q')\le C d^\delta,$$
and by {\bf Case 1} $d_{\omega_t}(p',q')\le C d_{\omega_0}(p',q')^\delta \le C d^\delta$. By triangle inequality we get  the desired estimate $d_{\omega_t}(p,q)\le C d^\delta$.

\medskip

\noindent $\bullet$ {\bf Case 3.} $\min(d_{\omega_0}(p,E), d_{\omega_0}(q,E) )\ge d$. Every point in the (Euclidean) line segment $\overline{pq}$ has norm of  $z_1$-coordinates no less than $d$, therefore 
$$d_{\omega_t} (p,q)\le L_{\omega_t} (\overline{pq})\le C d^{-(1-\delta)} L_{\omega_{\mathbb C^n}} (\overline{p,q})  = C d^\delta.$$

\medskip

Combining the all the cases above, we finish the proof of the lemma.

\end{proof}

Next we will prove the H\"older continuity of $\varphi_t$ with respect to $(X,\omega_0)$. To begin with, we first prove the gradient estimate of $\Phi:= (T-t)\dot \varphi + \varphi$ with respect to the evolving metrics $(X,\omega_t)$ (c.f. \cite{FGS}). 
\begin{lemma}\label{lemma 1.7}
There exists a uniform constant $C>0$ such that
\begin{equation*}
\sup_X |\nabla_{\omega_t} \Phi|_{\omega_t}\le C,\quad \forall ~ t\in [0,T).
\end{equation*}
\end{lemma}
\begin{proof}
Taking $\frac{\partial }{\partial t}$ on both sides of \eqref{eqn:MA}, we get
$$\frac{\partial}{\partial t} \dot\varphi = \Delta \dot\varphi + \frac 1 T \tr_\omega (\hat \omega_Y - \omega_0  )  =  \Delta \dot \varphi + \frac{1}{T-t}\tr_\omega \hat\omega_Y + \frac{1}{T-t}\Delta \varphi,   $$
where we used  the equation $ - \frac 1 T \tr_\omega \omega_0 = - \frac{n}{T-t} + \frac{t}{T(T-t)} \tr_\omega \hat \omega_Y + \frac{1}{T-t} \Delta \varphi  $.  Then we have the equation 
\begin{equation}\label{eqn:my 5}
(\frac{\partial}{\partial t} - \Delta )\Phi = \tr_\omega \hat \omega_Y - n\ge -n.
\end{equation}
By maximum principle, it follows that $\inf_X \Phi\ge - C$ for some constant depending also on $T$. Recall $\Phi$ is also bounded above by Lemma \ref{lemma 1.1}.  
And combining  \eqref {eqn:my 5} with Bochner formula the following equation holds:
\begin{equation*}
(\frac{\partial}{\partial t} - \Delta )|\nabla \Phi|^2_{\omega} = - |\nabla \nabla \Phi|^2 - |\nabla\bar \nabla \Phi| ^2 + 2 Re \innpro{ \nabla \Phi, \bar\nabla \tr_\omega \hat \omega_Y}. 
\end{equation*}
Fix a constant $B: = \sup_{X\times [0,T)}|\Phi| + 2$. By direct calculations the following equation holds
{\small
\begin{equation}\label{eqn:my 6}\begin{split}
(\frac{\partial}{\partial t} - \Delta ) \frac{|\nabla \Phi|^2}{ B - \Phi} & = \frac{\ho |\nabla \Phi|^2}{B - \Phi} + \frac{\abs{\nabla \Phi} \ho \Phi   }{(B-\Phi)^2} + 2 Re \innpro{ \nabla \log(B-\Phi), \bar \nabla \frac{\abs{\nabla \Phi}}{B-\Phi}   }\\
& = \frac{ - \abs{\nabla\nabla \Phi}  - \abs{\nabla \bar \nabla \Phi} + 2 Re \innpro{ \nabla \Phi, \bar \nabla  \tr_\omega \hat \omega_Y  }}{B-\Phi} + \frac{\abs{\nabla \Phi} ( \tr_\omega \hat \omega_Y - n  )}{(B-\Phi)^2}\\
& \quad  + 2 Re \innpro{ \nabla \log (B-\Phi), \bar \nabla \frac{\abs{\nabla \Phi}}{B- \Phi}  }.
\end{split}
\end{equation}
}
From the equation \eqref{eqn:my 1}, we have
\begin{equation}\label{eqn:my 7} \begin{split}
& \ho \frac{\tr_\omega \hat \omega_Y}{ B - \Phi}\\ \le~ & \frac{-c_0 |\nabla \tr_\omega \hat \omega_Y|^2 + C}{ B-\Phi} + \frac{\tr_\omega \hat \omega_Y (\tr_\omega \hat\omega_Y - n)}{(B-\Phi)^2} + 2 Re \innpro{ \nabla \log (B-\Phi), \bar \nabla \frac{\tr_\omega \hat\omega_Y}{ B - \Phi}  }.
\end{split}\end{equation}
Denote $$G = \frac{\abs{\nabla \Phi}}{ B - \Phi} + A \frac{\tr_\omega \hat \omega_Y}{ B - \Phi},\quad \text{where } A = 10 c_0^{-1}.$$
By \eqref{eqn:my 6}, \eqref{eqn:my 7} and Cauchy-Schwarz inequality we have 
\begin{align*}
& \ho G\\
\le ~ & \frac{ - \abs{\nabla \nabla \Phi} - \abs{\nabla \bar \nabla \Phi} - 9 \abs{\nabla \tr_\omega \hat \omega_Y  }    }{B-\Phi} + C G + C + 2 Re\innpro{\nabla \log(B-\Phi),\bar\nabla G}.
\end{align*}
Assuming the maximum of $G$ is attained at $(x_0,t_0)$, we may assume at this point $|\nabla \Phi|\ge A$, otherwise we are done yet. Then at this point $\ho G \ge 0$ and $\nabla G = 0$, hence we have $2 |\nabla \Phi|\cdot \nabla |\nabla \Phi|  = - G \nabla \Phi - A \nabla \tr_\omega \hat \omega_Y$. Taking norm on both side we get at $(x_0,t_0)$
\begin{equation}\label{eqn:my 8}
\frac{\abs{G \nabla \Phi + A \nabla \tr_\omega \hat\omega_Y  }}{2 \abs{\nabla \Phi}} = 2 \abs{\nabla |\nabla \Phi|}\le \abs{\nabla \nabla \Phi} + \abs{\nabla \bar \nabla \Phi},
\end{equation}
where we used the Kato's inequality in the last inequality. Therefore at $(x_0,t_0)$, we have 
{\small
\begin{align*}
0\le& (B-\Phi)^{-1} \bk{ - \frac{1}{2} G^2 + A G  \frac{|\nabla \tr_\omega \hat \omega_Y|}{|\nabla \Phi|} + \frac{A^2}{2\abs{\nabla \Phi}} \abs{\nabla \tr_\omega\hat \omega_Y} - 9 \abs{\nabla \tr_\omega \hat \omega_Y}     } + C G + C\\
\le & -\frac{G^2}{4 (B-\Phi)} + C G + C,
\end{align*}
}
so at $(x_0,t_0)$, $G \le C$. From this we get the desired bound on $|\nabla \Phi|$.
\end{proof}

An immediate consequence of the gradient bound is the uniform H\"older continuity of  $\varphi_t$ on $(X,\omega_0)$.
\begin{corr}\label{cor 1.1}
There exists a uniform constant $C>0$ such that 
\begin{equation*}
| \varphi_t(p) - \varphi_t(q)  |\le C d_{\omega_0}(p,q)^\delta,\quad \forall p,q\in X, \text{ and }\forall t\in [0,T).
\end{equation*}
\end{corr}
\begin{proof}
Recall the definition of $\Phi$, that
$$\Phi_t = (T-t) \dot\varphi + \varphi = (T- t)^2 \frac{\partial}{\partial t} \bk{ \frac{\varphi_t}{T-t}  }.$$ By the gradient bound in Lemma \ref{lemma 1.7} and distance estimate in Lemma \ref{lemma 1.6}, for any fixed points $p, q\in X$, we have
\begin{equation*}
| \Phi_t(p) - \Phi_t (q) |\le C d_{\omega_t}(p,q)\le C d_{\omega_0}(p,q)^\delta.
\end{equation*}
So
\begin{equation}\label{eqn:my 10}
\Big| \frac{\partial}{\partial t}\bk{ \frac{\varphi_t}{T-t}  }(p) - \frac{\partial}{\partial t}\bk{ \frac{\varphi_t}{T-t}  }(q)   \Big| \le \frac{C d_{\omega_0}(p,q)^\delta}{(T-t)^2},
\end{equation}
integrating \eqref{eqn:my 10} over $t\in [0,t_1)$ we get by noting that $\varphi_0 \equiv 0$ that
\begin{equation*}
\big|\frac{\varphi_{t_1}(p)}{T - t_1}  - \frac{\varphi_{t_1}(q)}{T- t_1}  \big|\le C d_{\omega_0}(p,q)^\delta \int_0^{t_1} \frac{1}{(T-t)^2}dt= C d_{\omega_0}(p,q)^\delta \frac{t_1}{T(T-t_1)},
\end{equation*}
cancelling the common factor $\frac{1}{T-t_1}$ on both sides we get the desired estimate since $t_1\in (0,T)$ is arbitrarily chosen.

\end{proof}
\begin{remark}
By an argument in \cite{Li}, the H\"older continuity of $\varphi_t$ implies that the distance functions satisfy the estimate in Lemma \ref{lemma 1.6}.
\end{remark}

Recall that the exceptional divisor $E$ is a $\mathbb{CP}^{k-1}$-bundle over $N$ and we identify $N$ with the zero section of this  bundle. Denote the bundle map by $\hat \pi: E \to N$ which is the restriction of $\pi: X\to Y$ to $E$.  From Corollary  \ref{cor 1.1}, we see that the limit $\varphi_T\in PSH(X,\hat\omega_Y)$ is also H\"older continuous in $(X,\omega_0)$.  Since $\hat\omega_Y |_{\hat\pi^{-1}(y)} = 0$ for any $y\in N$, we know that $\varphi_T|_{\hat\pi^{-1}(y)} = \text{const}$ for each $y\in N$ since $\hat \pi^{-1}(y)$ is connected. Thus $\varphi_T$ descends to a bounded function in $PSH(Y,\omega_Y)$, which we will still denote by $\varphi_T$. We shall show $\varphi_T$ is also H\"older continuous in $(Y,\omega_Y)$ with a possible different H\"older component.

\begin{lemma}\label{lemma 1.8}
There exists a uniform constant $C>0$ such that 
\begin{equation}\label{eqn:my 11}|\varphi_T(p) - \varphi_T(q)  |\le C d_{\omega_Y}(p,q)^{\delta_Y},\quad \forall ~ p,q\in Y,\end{equation}
where $\delta_Y = \min\{\delta (1-\delta),\delta^2\}\in  (0,1)$.
\end{lemma}
\begin{proof}
We denote the zero section of the $\mathbb{CP}^{k-1}$-bundle $\hat \pi: E\to N$ by $\hat N$, and it is well-known that $\hat N \cong N$. It suffices to show \eqref{eqn:my 11} for $p,q$ in a fixed tubular neighborhood $T(N)$ of $N$, since on $Y\backslash T(N)$ the metric $\pi^*\omega_Y = \hat \omega_Y$ is equivalent to $\omega_0$, and the estimate follows from Corollary \ref{cor 1.1}. 

Choose coordinates charts $(V_\alpha, w_{\alpha, j})$ covering  $T(N)$ such that $V_\alpha \cap N = \{ w_{\alpha,1} = \cdots = w_{\alpha, k} = 0  \}$. We also assume that any $p,q\in T(N)$ with $d_{\omega_Y}(p,q)\le 1$ lie in the same $V_\alpha$, if the charts are chosen sufficiently fine.  We will work in a fixed chart $(V, w_i)$, and omit the subscript $\alpha$. On this open set $\omega_Y$ is equivalent to the Euclidean metric $\omega_{\mathbb C^n}$ in $(\mathbb C^n, w_i)$, so without loss of generality, we may assume $\omega_Y = \omega_{\mathbb C^n}$ on $V$. The map $\pi: U\to V$ can be locally expressed as \begin{equation}\label{eqn:map pi}w_1 = z_1,\, w_2  = z_1 z_2,\cdots, w_k = z_1 z_k,\, w_{k+1} = z_{k+1},\cdots, w_n = z_n,\end{equation}
where $(U,z_i)$ is an open chart on $X$. The zero section $\hat N$ can be locally expressed as $\hat N\cap U = \{z_1=\cdots = z_k = 0\}$.

We consider different cases depending on the positions of $p,q$ in $V$. Denote $0<d= d_{\omega_Y}(p,q) \le 1/4$.

\medskip

\noindent $\bullet$ {\bf Case 1:} we assume $p,q\in N$.  Take the unique pre-images under $\hat p$ of $p,q$ in $\hat N$, $\hat p, \hat q$, respectively. We know that $\varphi_T(p) = \varphi_T(\hat p)$ and $\varphi_T(q) = \varphi_T(\hat q)$. The line segment $\overline{pq}$ is contained in $N$ and similarly $\overline{\hat p \hat q}$ is contained in $\hat N$ as well. From the local expressions \eqref{eqn:pullback} and \eqref{eqn:pullback 1} of $\omega_X :=\pi^*\omega_Y + \varepsilon_0 \ddbar \log \sigma_X$, we conclude that $d_{\omega_Y}(p,q) = L_{\omega_Y}(\overline{pq})$ is comparable to $L_{\omega_X} (\overline{\hat p\hat q})$, which is no less than $c_1 d_{\omega_0}(\hat p,\hat q)$, for some uniform $c_1>0$. Therefore
\begin{align*}
| \varphi_T(p) - \varphi_T(q)  | = & |\varphi_T(\hat p) - \varphi_T(\hat q)| \le C d_{\omega_0}(\hat p,\hat q)^\delta\le C d_{\omega_Y}(p,q)^\delta,
\end{align*}
as desired.

\medskip

\noindent $\bullet$ {\bf Case 2:} if $0<\min\{d_{\omega_Y}(p,N), d_{\omega_Y} (q,N)   \}\le 2 d^{1-\delta}$. Take the orthogonal projections of $p$ and $q$ to $N$, $p', q'$ respectively. In other words, $p'$ ($q'$ resp.) has the same $(w_{k+1},\ldots, w_n)$-coordinates as $p$ ($q$ resp.) but the first $k$-coordinates are zero. From the assumption we know that $d_{\omega_Y}(p, p') = L_{\omega_Y}(\overline{pp'})\le 3d^{1-\delta}$ and $d_{\omega_Y}(q, q') = L_{\omega_Y}(\overline{q q'})\le 3d^{1-\delta}$. The pulled-back of the line segment $\overline{pp'}$ under $\pi$ is also a line segment $\overline{\pi^{-1}({p}) \hat p'}$ in $(U,z_i)$ connecting $\pi^{-1}(p)$ and a unique point $\hat p'\in \hat \pi ^{-1}(p')\subset E$, and $\hat p' = (0,\frac{w_2}{w_1},\ldots, \frac{w_k}{w_1},w_{k+1},\ldots, w_n)$, where $w_j$ denotes the $w_j$-coordinate at $p$. It holds that $\varphi_T(p') = \varphi_T(\hat p')$ since $\hat p'$ lies at the fiber over $p'$. Again from the local expressions  \eqref{eqn:pullback} and \eqref{eqn:pullback 1} of $\omega_X$, it follows that $L_{\omega_X} (\overline{\pi^{-1}(p) \hat p'})$ is comparable to the length of $\overline{p p'}$ under $\omega_Y$, therefore
$$d_{\omega_0}(\pi^{-1}(q), \hat p')\le C L_{\omega_X} ( \overline{ \pi^{-1}(p) \hat p'  }  )\le C L_{\omega_Y}(\overline{p p'})\le C d^{1-\delta},$$
from which we derive that
\begin{equation*}
|\varphi_T(p) - \varphi_T(p') |  = |\varphi_T(\pi^{-1}(p)) - \varphi_T(\hat p ')  |\le C d_{\omega_0}(\pi^{-1}(p), \hat p '  )^\delta \le C d^{\delta_Y}.
\end{equation*}
Similar estimate also holds for $|\varphi_T(\pi^{-1}(q)) - \varphi_T(q')  |$.  Since $p',q'\in E$ and $d_{\omega_Y}(p',q')\le d$, by {\bf Case 1} we also have $|\varphi_T(p') - \varphi_T(q')|\le C d^\delta$. The desired estimate \eqref{eqn:my 11} in this case then follows from triangle inequality.

\medskip

\noindent$\bullet$ {\bf Case 3:} $\min\{d_{\omega_Y} (p,N), d_{\omega_Y}(q, N)\}> 2d^{1-\delta}$.  The line segment $\gamma(s) = \overline{pq}$ is strictly away from $N$, in fact, $\sigma_Y (\gamma(s) ) \ge d^{2(1-\delta)}$. Therefore the pulled-back $\hat \gamma(s) = \pi^{-1}(\gamma(s))$ joins $\pi^{-1}(p)$ to $\pi^{-1}(q)$ and $\sigma_X ( \hat\gamma(s)  ) \ge d^{2(1-\delta)}$.  From \eqref{eqn:my 3} that $\omega_X \le C \frac{\pi^*\omega_Y}{ \sigma_X}$ on $X\backslash E$ we have 
$$d_{\omega_0}(\pi^{-1}(p), \pi^{-1}(q)) \le C L_{\omega_X} (\hat\gamma) \le \frac{C}{d^{1-\delta}} L_{\omega_Y}(\gamma) \le C d ^{\delta}.$$
Therefore
{\small
\begin{equation*}
| \varphi_T(p) - \varphi_T(q)  |  = | \varphi_T(\pi^{-1}(p)) - \varphi_T(\pi^{-1}(q))  |\le C d_{\omega_0}\big( \pi^{-1}(p), \pi^{-1}(q)   \big)^\delta\le C d^{\delta^2}\le C d^{\delta_Y},
\end{equation*}
}as desired.

Combining the cases discussed above, we finish the proof of the lemma.

\end{proof}

The positive $(1,1)$-form $\omega_T = \omega_Y  + \ddbar \varphi_T$ defines a K\"ahler metric $g_T$ on $Y\backslash N$, with the associated function $\tilde d_T: Y\backslash N \times Y\backslash N \to [0,\infty)$ defined by 
\begin{equation*}
\tilde d_T(p,q): = \inf \Big\{ \int_{\gamma\backslash N} \sqrt{g_T( \gamma',\gamma'  )} |~ \gamma\subset Y \text{ and $\gamma$ joins $p$ to $q$}    \Big\}
\end{equation*}
for any $p,q\in Y\backslash N$ and $\gamma$ is taken over all piecewise smooth curves in $Y$ with only finitely many intersections with $N$. With this distance function, $(Y\backslash N, \tilde d_T)$ becomes a metric space, which may not be complete. We want to extend the distance function to the whole $Y$. To begin with, we need a trick from \cite{Li}. 
\begin{lemma}\label{lemma 1.9}
There exists a uniform constant $C>0$ such that for any $p\in Y\backslash N$ and $r_p =  d_{\omega_Y}(p, N)>0$
\begin{equation*}
\tilde d_T(p,q)\le C d_{\omega_Y}(p,q)^{\delta_Y/2},\quad \forall q\in B_{\omega_Y}(p,r_p/2).
\end{equation*}

\end{lemma}
\begin{proof}
The ball $B:=B_{\omega_Y}(p, r_p/2)$ is strictly away from $N$ so $\omega_T$ is smooth on $B$. The function $d_p(x) = \tilde d_T(p, x)$ is Lipschtiz continuous and satisfies $|\nabla d_p|_{\omega_T}\le 1$ a.e.. For any $r\le \frac{r_p}{2}$, we have 
\begin{align*}
\int_{B_{\omega_Y}(p,r)} |\nabla d_p|_{\omega_Y}^2 \omega_Y^n \le &\int_{B_{\omega_Y}(p,r)} |\nabla d_p|_{\omega_T}^2 (\tr_{\omega_Y} \omega_T )\omega_Y^n \\
\le & \int_{B_{\omega_Y}(p,r)} ( n + \Delta_{\omega_Y} \varphi_T  ) \omega_Y^n\\
\le & C r^{2n} + \int_{B_{\omega_Y}(p, 1.5 r)} |\varphi_T(x) - \varphi_T(p)| |\Delta_{\omega_Y} \eta| \omega_Y^n\\
\le & C r^{2n} + C r^{\delta_Y + 2n - 2} \le C r^{2n - 2 + \delta_Y},
\end{align*}
where $\eta$ is a standard cut-off function supported in $B_{\omega_Y}(p, 1.5 r)$ and identically equal to $1$ on $B_{\omega_Y}(p,r)$, and it satisfies $|\Delta_{\omega_Y} \eta|\le C r^{-2}$. Then by Poincare inequality and Campanato's lemma (see Theorem 3.1 in \cite{HL}) we get
$$\tilde d_T(p,q) = d_p(q) = | d_p(q) - d_p(p)  |\le C d_{\omega_Y}(p,q)^{\delta_Y/2},$$
for any $q\in B_{\omega_Y}(p, r_p/2)$.

\end{proof}

\begin{lemma}\label{lemma 1.10}
There exist constants $C>0$ and $\delta_0\in (0,1)$ such that
\begin{equation}\label{eqn:my 12}
\tilde d_{T}(p,q)\le C d_{\omega_Y}(p,q)^{\delta_0},\quad \forall p, q\in Y\backslash N.
\end{equation}
\end{lemma}
\begin{proof}
We use the same notation as in the proof of Lemma \ref{lemma 1.8}. It suffices to show \eqref{eqn:my 12} for $p,q\in V$ where $V$ is a fixed coordinate chart in $Y$ and recall locally the map $\pi: U\to V$ is given by \eqref{eqn:map pi}. Let $d = d_{\omega_Y}(p,q)<1/4$. 

In case $\min\{d_{\omega_Y}(p,N), d_{\omega_Y}(q,N)\}> 2d$, then $q\in B_{\omega_Y}(p, \frac 12 d_{\omega_Y}(p,N) )$. By Lemma \ref{lemma 1.9}, it follows that $\tilde d_T(p,q)\le C d^{\delta_Y/2}$. So it only remains to consider the case when the minimum above is $\le 2d$. Let  $p',q'\in N\cap V$ be the orthogonal projection (assuming $\omega_Y = \omega_{\mathbb C^n}$) of $p,q$ to $N$, respectively.  Then $\max\{d_{\omega_Y}(p, p '), d_{\omega_Y}(q,q')\} \le 3d$ and $d_{\omega_Y}(p',q')\le d$. Choose the unique pre-images $\hat p,\hat q\in \hat N\subset E$ in the zero section $\hat N$ of the bundle $\hat \pi: E\to N$, of $p', q'$, i.e. $\hat \pi(\hat p) = p'$ and $\hat \pi(\hat q) = q'$.
From the local expressions \eqref{eqn:pullback} and \eqref{eqn:pullback 1} of $\omega_X = \pi^*\omega_Y + \varepsilon_0\ddbar \log \sigma_X$, it can be shown that $d_{\omega_X}(\hat p,\hat q) \le C d_{\omega_Y} (p',q')\le C d$. As in the proof of {\bf Case 1} in Lemma \ref{lemma 1.6} with $p,q$ in that lemma replaced by $\hat p, \hat q$ here. Recall that the piecewise line segment $\gamma$ which connects $\hat p$ and $\hat q$ lies outside $E$, except the two end points. Furthermore $\gamma$ is chosen independent of $t\in [0,T)$ and we have
\begin{equation}\label{eqn:my 13}
\int_\gamma \sqrt{ g_t( \gamma',\gamma'  )  }  = L_{\omega_t}(\gamma) \le C d_{\omega_X}(\hat p,\hat q)^\delta\le C d^\delta,
\end{equation}
since $g_t \to g_T$ (locally) smoothly on $\gamma\backslash \{\hat p, \hat q\}$, letting $t\to T^-$ and applying Fatou's lemma to \eqref{eqn:my 13}, we get
\begin{equation*}
\int_\gamma\sqrt{ g_T( \gamma',\gamma' )  } \le C d^{\delta}.
\end{equation*}
Denote the image curve $\gamma_0 = \pi(\gamma)\subset Y$ which joins $p'$ to $q'$ and is contained in $Y\backslash N$ except the end points. It follows then that $L_{\omega_T}(\gamma_0)\le C d^\delta$. The line segment $\gamma_1(s) = \overline{p p'}$ is given by 
$$\gamma_1(s) = ( s w_1(p), \cdots, s w_k(p),  w_{k+1} (p),\cdots, w_n(p)  ) ,\quad s\in [0,1]$$
and its pulled-back to $X$, $\hat \gamma_1(s) = \pi^{-1}(\gamma(s))$ is locally given by 
$$\hat \gamma_1(s) = ( s w_1(p), \frac{w_2(p)}{w_1(p)}, \cdots,\frac{w_{k}(p)}{w_1(p)}, w_{k+1}(p),\ldots, w_{n}(p)   ),\quad s\in [0,1].$$
By the estimate in Lemma \ref{lemma 1.5}, it follows that
{\small
\begin{equation*}
\int_{\hat \gamma_1}\sqrt{ g_t( \hat \gamma_1',\hat \gamma_1 '  )  } \le C \int_{\hat \gamma_1} \sqrt{ \frac{g_X(\hat\gamma_1',\hat \gamma_1')}{s^{2(1-\delta)} |\mathbf{w}(p)|^{2(1-\delta)}   }   } \le C \int_{\hat \gamma_1} \frac{\sqrt{\pi^*\omega_Y( \hat \gamma_1',\hat \gamma_1 '   )}}{ s^{1-\delta} |\mathbf{w}(p)|^{1-\delta}  } \le C |\mathbf{w}(p)|^\delta\le C d^\delta,
\end{equation*}
}
where $\mathbf{w}(p) =  ( w_1(p),\cdots, w_k(p))$. By Fatou's lemma and letting $t\to T^-$, we get
\begin{equation*}
L_{\omega_T}(\gamma_1) = \int_{\hat \gamma_1} \sqrt{ g_T( \hat \gamma_1' , \hat \gamma_1' )   } \le C d^\delta.
\end{equation*} Similarly the line segment $\gamma_2 = \overline{q q'}$ also have $L_{\omega_T}(\gamma_2)\le C d^\delta$. Now we define a piecewise smooth curve $$\bar \gamma = \gamma_1 + \gamma_0 + \gamma_2,   $$
which joins $p$ to $q$ and lies entirely outside $N$, except the two points $p'$ and $q'$. And combining the estimates above we get
$$L_{\omega_T}(\bar \gamma) = L_{\omega_T}(\gamma_1) + L_{\omega_T}(\gamma_0) + L_{\omega_T}(\gamma_2)\le C d^\delta. $$
Then by definition $$\tilde d_T (p,q)\le L_{\omega_T} (\bar \gamma) \le C d^\delta = C d_{\omega_Y}(p,q)^\delta. $$

\medskip

From the discussions above, \eqref{eqn:my 12} follows for $\delta_0 = \min(\delta_Y/2, \delta)$.
\end{proof}

We now extend the distance function $\tilde d_T$ to $Y$, for any $p\in Y\backslash N$ and $q\in N$, we define the distance
\begin{equation}\label{eqn:dT}
d_T(p,q):= \lim_{i\to\infty} \tilde d_T(p, q_i),
\end{equation}
where $\{q_i\}\subset Y\backslash N$ is a sequence of points such that $d_{\omega_Y} (q,q_i)\to 0$. We need to justify $d_T$ is well-defined, i.e. the limit exists and is independent of the choice of the sequence $\{q_i\}$.
\begin{lemma}\label{lemma 1.11}
The limit in \eqref{eqn:dT} exists and for any other sequence $\{q_i'\}\subset Y\backslash N$ converging to $q$ in $(Y,\omega_Y)$, the following holds
\begin{equation*}
\lim_{i\to \infty} \tilde d_T(p,q_i) = \lim_{i\to \infty} \tilde d_T(p, q_i').
\end{equation*}

\end{lemma}
\begin{proof}
This is in fact an immediate consequence of Lemma \ref{lemma 1.10}. Observe that
\begin{equation*}
  | \tilde d_T(p,q_i) - \tilde d_T(p,q_j)   | \le \tilde d_T(q_i, q_j) \le C d_{\omega_Y}(q_i,q_j)^{\delta_0}\to 0,\quad \text{as }i,j\to\infty.
\end{equation*}
Thus $\{\tilde d_T(p,q_i)\}_{i=1}^\infty$ is a Cauchy sequence hence it converges. On the other hand, similarly we have
$$|\tilde d_T(p,q_i) - \tilde d_T(p,q_i')  |\le C d_{\omega_Y}(q_i,q_i')^\delta\to 0,\quad \text{as }i\to \infty, $$
and it then follows that the limit is independent of the choice of $\{q_i\}$ converging to $q$.

\end{proof}
We then define the distance between points in $N$ as follows: for any $p,q\in N$
\begin{equation*}
d_T(p,q): = \lim_{i\to \infty} \tilde d_T ( p_i, q_i ),
\end{equation*}
for two sequences $Y\backslash N\supset\{p_i\}\to p$ and $Y\backslash N\supset \{q_i\}\to q$ under $d_{\omega_Y}$. It can be checked similar as Lemma \ref{lemma 1.11} that the limit exists and is independent of the choice of sequences converging to $p$ or $q$. Thus $(Y, d_T)$ defines a compact metric space, since $d_T(p,q)\le C d_{\omega_Y}(p,q)^{\delta_0}$ for any $p,q\in Y$, which follows from Lemma \ref{lemma 1.10}. 


\smallskip

We now turn to the Gromov-Hausdorff convergence of the flow.  The proof is motivated by \cite{RZ} (see also \cite{TWY, GTZ,FGS}).

\begin{lemma}
For any $t_i\to T^-$, there exists a subsequence which we still denote by $\{t_i\}$ such that as compact metric spaces $$(X,\omega_{t_i})\xrightarrow{d_{GH}} (Z,d_Z)  $$for some compact metric space $(Z,d_Z)$.

\end{lemma}
\begin{proof}
For any $\epsilon>0$, we choose an $\epsilon$-net $\{x_j\}_{j=1}^{N_{i,\epsilon}}\subset (X,\omega_{t_i})$, in the sense that $d_{\omega_{t_i}}(x_j, x_{j'})>\epsilon$ and the open balls $\{B_{\omega_{t_i}}(x_j, 2\epsilon)\}_{j}$ cover $(X,\omega_{t_i})$. From Lemma \ref{lemma 1.6}, we have
$$\epsilon < d_{\omega_{t_o}} (x_j, x_{j'}) \le C d_{\omega_0} (x_j, x_{j'})^\delta,$$ thus under the fixed metric $d_{\omega_0}$, each pair of points $(x_j, x_{j'})$ from the $\epsilon$-net has distance at least $C^{-1/\delta} \epsilon^{1/\delta}$, thus the balls $\{B_{\omega_0}(x_j, C^{-1/\delta} \epsilon^{1/\delta}/2)\}_j $ are disjoint, so for some $c>0$
$$ N_{i,\epsilon} c \epsilon^{1/\delta^{2n}}   =  \sum_{j=1}^{N_{i,\epsilon}} c \epsilon^{1/\delta^{2n}}    \le \int_{\cup_j B_{\omega_0}(x_j, C^{-1/\delta} \epsilon^{1/\delta}/2)} \omega_0^n \le \int_X \omega_0^n,$$ from which we derive an upper bound of $N_{i,\epsilon}\le N_\epsilon$, which is independent of $i$. Then by Gromov's precompactness theorem (\cite{Gr}), there exists a compact metric space $(Z,d_Z)$, such that up to a subsequence $(X,\omega_{t_i})\xrightarrow{d_{GH}} (Z,d_Z)$.

\end{proof}

\begin{lemma}\label{lemma 1.13}
There exists an open and dense subset $Z^\circ \subset Z$ such that $(Z^\circ, d_Z)$ and $(Y\backslash N, d_T)$ are homeomorphic and locally isometric.
\end{lemma}

\begin{proof}
For notation convenience we denote $Y^\circ = Y\backslash N$. The maps $\pi_i = \pi: (X,\omega_{t_i}) \to (Y, \omega_Y)$ are Lipschitz by the estimate $\pi^*\omega_Y \le C \omega_{t_i}$ as in (ii) of Lemma \ref{lemma 1.2}. The target space $(Y,\omega_Y)$ is compact, so by Arzela-Ascoli theorem up to a subsequence of $\{t_i\}$, along the GH convergence $(X,\omega_{t_i})\xrightarrow{d_{GH}}(Z,d_Z)$, the maps $\pi_i \xrightarrow{{GH}} \pi_Z$, for some map $\pi_Z : (Z,d_Z)\to (Y,\omega_Y)$, in the sense that for any $(X,\omega_{t_i})\ni x_i\xrightarrow{d_{GH}} z\in Z$, $\pi_i(x_i)\xrightarrow{d_{\omega_Y}} \pi_Z(z)$ in $Y$. $\pi_Z$ is also Lipschitz from $(Z,d_Z)$ to $(Y,\omega_Y)$, i.e. $d_{\omega_Y}(\pi_Z(z_1),\pi_Z(z_2))\le C d_Z(z_1, z_2)$ for any $z_1,z_2\in Z$.  We denote $Z^\circ = \pi_Z^{-1}(Y^\circ)$, and we will show that $\pi_Z|_{Z^\circ}: (Z^\circ, d_Z) \to (Y^\circ, d_T)$ is homeomorphic and locally isometric, and $Z^\circ \subset Z$ is open and dense. The openness of $Z^\circ \subset Z$ follows from the continuity of the map $\pi_Z : (Z,d_Z)\to (Y,d_{\omega_Y})$ and the fact that $Y^\circ\subset Y$ is open.

\medskip

\noindent $\bullet$ {\bf $\pi_Z|_{Z^\circ}$ is injective:} suppose $z_1, z_2\in Z^\circ = \pi_Z^{-1}(Y^\circ)$ are mapped to the same point $y\in Y^\circ$, $\pi_Z(z_1) = \pi_Z(z_2) = y$. Since $(Y^\circ, \omega_T)$ is an incomplete smooth Riemannian manifold and locally in $Y^\circ$, $d_T$ is induced from the Riemannian metric, we can find a small $r = r_y>0$ such that the metric ball $(B_{\omega_T}(y, 2r), \omega_T)$ is geodesically convex. Choose two sequence of points $z_{1,i},z_{2,i}\in (X,\omega_{t_i})$ converging in GH sense to $z_1, z_2\in Z$, respectively.   From the convergence of $\pi_i\xrightarrow{GH} \pi_Z$, we obtain $d_{\omega_Y} ( \pi_i(z_{1,i}), \pi_Z(z_1)   )\xa{i\to\infty} 0$ and $d_{\omega_Y} ( \pi_i(z_{2,i}), \pi_Z(z_2)   )\xa{i\to\infty} 0$. By Lemma \ref{lemma 1.10}, the same limits hold with $d_{\omega_Y}$ replaced by $d_T$. In particular this implies that $d_T(\pi_i(z_{1,i}), \pi_{i}(z_{2,i})  )\xa{i\to\infty} 0$ and both $\pi_{i}(z_{1,i})$ and $\pi_i(z_{2,i})$ lie inside $B_{\omega_T}(y, r/2)$ when $i$ is large enough. We can find $\omega_T$-geodesics $\gamma_i\subset B_{\omega_T} (y, r)$ connecting $\pi_i(z_{1,i})$ and $\pi_i(z_{2,i})$, and by the uniform and smooth convergence of $\omega_{t_i}\to \omega_T$ on $\overline{\pi^{-1}(B_{\omega_T} (y, 2r)   )  }$, it follows that
\begin{equation*}
0\le d_{\omega_{t_i}}(z_{1,i}, z_{2,i}) \le L_{\omega_{t_i}} (\hat \gamma_i) \le L_{\omega_T}(\gamma_i) + \epsilon_i = d_{T}( \pi(z_{1,i}),\pi_i( z_{2,i})  ) + \epsilon_i \xa{i\to\infty} 0,
\end{equation*}
where $\hat \gamma_i = \pi^{-1}(\gamma_i)$ is a curve joining $z_{1,i}$ to $z_{2,i}$ and $\{\epsilon_i\}$ is a sequence tending to zero. From the definition of GH convergence we see that $$d_Z(z_1, z_2) = \lim_{i\to\infty}d_{\omega_{t_i}} (z_{1,i}, z_{2,i}) = 0.  $$
Hence $z_1 = z_2$ and $\pi_{Z}|_{Z^\circ}$ is injective. 

\medskip

\noindent$\bullet$ {\bf $\pi_{Z} |_{Z^\circ}: (Z^\circ, d_Z) \to (Y^\circ, d_T)$ is a local isometry.} We first explain what the local isometry means. It says that for any $z\in Z^\circ$ and $y = \pi_Z(z)\in Y^\circ$, we can find open sets $z\in U\subset Z^\circ$ and $y\in V\subset Y^\circ$ such that $\pi_Z|_U: (U,d_Z) \to (V,d_T)$ is an isometry.

  There exists a small $r = r_y>0$ such that the metric ball $(B_{\omega_T}(y, 3r),\omega_T) \subset Y^\circ$ and is geodesically convex. Take $U = (\pi_Z|_{Z^\circ})^{-1} (B_{\omega_T}(y, r)  )$. Since $B_{\omega_T}(y, r)$ is also open in $(Y, \omega_Y)$, it can be seen that $U$ is open in $Z^\circ$ and is a neighborhood of $z\in Z^\circ$. We will show $\pi_Z|_{U}: (U,d_Z) \to ( B_{\omega_T}(y,r),\omega_T  )$ is an isometry, i.e. for any $z_1,z_2\in U$, and $y_1 = \pi_Z(z_1)$, $y_2 = \pi_Z(z_2)$, we have $d_Z(z_1, z_2) = d_T(y_1, y_2)$.

\smallskip

We choose sequences of points $z_{1,i},z_{2,i}\in (X,\omega_{t_i})$ converging in GH sense to $z_1, z_2$, respectively, as before. It then follows from $\pi_i \xa{GH} \pi_Z$ and Lemma \ref{lemma 1.10} that $d_{T}( \pi_i(z_{a,i}), y_a  ) \to 0$ as $i\to \infty$, for each $a= 1,2$. In particular when $i$ is large enough, $\pi_i(z_{a,i})\in B_{\omega_T}(y, 1.1 r)$. Choose a minimal $\omega_{t_i}$-geodesic $\hat\gamma_i$ joining $z_{1,i}$ to $z_{2,i}$, and we have $$d_{\omega_{t_i}}(z_{1,i}, z_{2,i}) = L_{\omega_{t_i}}(\hat \gamma_i) \xa{i\to \infty } d_Z(z_1, z_2). $$ Denote the image $\gamma_i = \pi_i(\hat\gamma_i)$ which is a continuous curve joining $\pi_i(z_{1,i})$ to $\pi_i(z_{2,i})$. If $\gamma_i\subset B_{\omega_T}(y, 3r)$ (for a subsequence of $i$), since $\omega_{t_i} $  converge smoothly and uniformly to $\omega_T$ on the compact subset $\overline{\pi^{-1} ( B_{\omega_T} (y, 3r) )   }$,  it follows 
\begin{equation*}
d_{T}( \pi_i(z_{1,i}), \pi_i (z_{2,i})   ) \le L_{\omega_T} (\gamma_i) \le L_{\omega_{t_i}} (\hat \gamma_i) + \epsilon_i \xa{i\to\infty} d_Z(z_1,z_2).
\end{equation*}
In case $\gamma_i\not\subset B_{\omega_T}(y, 3 r)$ for $i$ large enough, we have
\begin{equation*}
d_T( \pi_i(z_{1,i}), \pi_i (z_{2,i})   ) \le 2.5 r \le L_{\omega_T}( \gamma_i\cap B_{\omega_T} (y, 3r)  ) \le L_{\omega_{t_i}} (\gamma_i) + \epsilon_i \xa{i\to\infty} d_Z(z_1, z_2).
\end{equation*}
Observe that $d_T( \pi_i(z_{1,i}), \pi_i (z_{2,i})   )\xa{i\to \infty} d_{T} (\pi_Z(z_1) , \pi_Z(z_2)  ) = d_T(y_1, y_2) $. So by the discussion in both cases, it follows that $d_T(y_1, y_2)\le d_Z(z_1, z_2)$. To see the reverse inequality, by the geodesic convexity of $(B_{\omega_T}(y, 3r), \omega_T  )$, we can find minimal $\omega_T$-geodesics $\sigma_i\subset B_{\omega_T}(y, 3r)$ connecting  $\pi_i(z_{1,i})$ and $\pi_i(z_{2,i})$ for $i$ large enough. The pulled-back $\hat \sigma_i = \pi^{-1}(\sigma_i)\subset \overline{B_{\omega_T}(y, 3r)  }$ joins $z_{1,i}$  to $z_{2,i}$, again by the local smooth convergence of $\omega_{t_i}$ to $\omega_T$, we have 
\begin{equation*}
d_{\omega_{t_i}}(z_{1,i}, z_{2,i}) \le L_{\omega_{t_i}} ( \hat \sigma_i )\le L_{\omega_T}( \sigma_i  ) + \epsilon_i = d_{T}(\pi_i(z_{1,i}), \pi_i(z_{2,i})  ) + \epsilon_i\xa{i\to\infty} d_{T} (y_1,y_2),
\end{equation*}
letting $i\to \infty$ we get $d_Z(z_1, z_2)\le d_T(y_1, y_2)$. Thus we show that $d_{Z}(z_1,z_2) = d_T(y_1, y_2)$, as desired.

\medskip

\noindent $\bullet$ {\bf $\pi_Z|_{Z^\circ}$ is surjective.} This follows from definition. Indeed, for any $y\in Y^\circ$, take $z = z_i = \pi^{-1}(y)\in (X,\omega_{t_i})$, up to a subsequence $z_i\xa{d_{GH}} z_0\in Z$. Since $\pi_i \xa{GH} \pi_Z$, we get $d_{\omega_Y} (y, \pi_Z(z_0)) = d_{\omega_Y} ( \pi_i(z_i),  \pi_Z(z_0) ) \to 0$ as $i\to \infty$. So $\pi_Z(z_0) = y$ and $z_0\in Z^\circ$ is the pre-image of $y$ under $\pi_Z|_{Z^\circ}$.

\medskip

Combining the discussions above, we see that $\pi_Z|_{Z^\circ}: (Z^\circ, d_Z) \to (Y^\circ, d_T)$ is a bijection and thus a homeomorphism (noting that the continuity of the maps $\pi_Z|_{Z^\circ}$ and $(\pi_Z|_{Z^\circ})^{-1}$ follow from the local isometry property). 

\medskip

It only remains to show $Z^\circ \subset Z$ is dense. Suppose not, there exists a point $z_0\in Z$ such that $B_{d_Z} (z_0,\bar \varepsilon) \subset Z\backslash Z^\circ$ for some $\bar \varepsilon>0$. Choose a sequence of points $x_{i}\in (X,\omega_{t_i})$ such that $x_{i}\xa{d_{GH}} z_0$. We claim that $d_{\omega_{t_i}}(x_{i}, E) \to 0$ as $i\to \infty$, where $E$ is the exceptional divisor of the blown-down map $\pi: X\to Y$. If not, then $d_{\omega_{t_i}}(x_{i}, E)\ge a_0>0$ for a sequence of large $i$'s, by Lemma \ref{lemma 1.6}, under the fixed metric $w_0$, $d_{\omega_0} (x_{i}, E)\ge C^{-1/\delta} a_0^{1/\delta}>0$, thus $\{x_{i}\} \subset K$, for some compact subset $K\Subset X\backslash E$. It then follows that $\pi_i(x_{i})\in \pi(K)\Subset Y^\circ$, and this contradicts the fact that $d_{\omega_Y}( \pi_i(x_{i}), \pi_Z(z_0)  )\to 0$ and $\pi_Z(z_0)\not\in Y^\circ$. Therefore, we may assume without loss of generality that $x_{i}\in E$ for all $i$. Moreover, from Lemma \ref{lemma SW} below, we may replace $x_{i}\in E$ by the point in the same fiber  as $x_{i}$ of the $\mathbb{CP}^{k-1}$-bundle  $\hat \pi: E\to N$
and the zero section $\hat N$. So we can assume in addition that $x_{i}\in \hat N$. Denote the points $y_{i} = \pi_i(x_i)\in N$ and $y_0 = \pi_Z(z_0)\in N$.  From $\pi_i\xa{GH} \pi_Z$ and $x_i\xa{d_{GH}} z_0$, we have $d_{\omega_Y}(y_i, y_0)\to 0$ as $i\to \infty$.

We may choose a coordinates chart $(V, w_j)$ as before, which is centered at $y_0$ and contains all but finitely many $y_i$, and $N\cap V = \{w_1 = \cdots = w_k = 0\}$. We take an open set $(U,z_j)$ over $(V,w_j)$, such that the map $\pi: U\to V$ is expressed as in \eqref{eqn:map pi}. We fix a point $p\in V\backslash N$ whose $w$-coordinate is $w(p) = (r, 0,\cdots, 0)$ for some $r>0$ to be determined. Take $\hat p = \pi^{-1}(p)$ and its $z$-coordinate is $z(\hat p) = (r,0,\cdots, 0)$. The point(s) $\hat p_i = \hat p\in (X,\omega_{t_i})$ converge (up to a subsequence) in GH sense to some point $p_Z\in Z$, and as above, we have $d_{\omega_Y}(p, \pi_Z(p_Z)) = d_{\omega_Y}(\pi_i(\hat p_i), \pi_Z(p_Z)) \to 0$ as $i\to\infty$,  so $p = \pi(p_Z)\in Y^\circ$ and $p_Z\in Z^\circ$. From the assumption we have $d_Z(z_0, p_Z)\ge \bar \varepsilon>0$. On the other hand, by the local expressions \eqref{eqn:pullback} and \eqref{eqn:pullback 1} of $\omega_X = \pi^*\omega_Y + \varepsilon_0 \ddbar \log \sigma_X$, we find that line segments $\overline{\hat p \hat z_0} + \overline{\hat z_0 x_i}$ in $(U,z_j)$ have $\omega_X$-length $\le C r + \epsilon_i$ for some sequence $\epsilon_i\to 0$, where we denote $\hat z_0 = \pi^{-1}(p_0)\cap \hat N$, i.e. $\hat z_0$ is the origin in $(U,z_j)$. So $d_{\omega_0}(\hat p_i, x_i)\le C (r + \epsilon_i)$ and by Lemma \ref{lemma 1.6}, $d_{\omega_{t_i}}(\hat p_i, x_i)\le C(r + \epsilon_i  )^\delta$. Letting $i\to\infty$ we get $d_Z(p_Z, z_0)\le C r^\delta$. If we choose $r$ small such that  $C r^{\delta} = \bar \varepsilon /2$, we would get a contradiction. Therefore $Z^\circ\subset Z$ is dense.

\end{proof}

By exactly the same proof of Lemma 3.2 in \cite{SW1}, we have 
\begin{lemma}\label{lemma SW}
There is a uniform constant $C>0$ such that 
\begin{equation*}
\mathrm{diam}( \hat \pi^{-1}(y), \omega_t  )\le C (T-t)^{1/3},\quad \forall t\in [0,T), \text{ and }\forall y\in N.
\end{equation*}
\end{lemma}
That is to say, the diameters of the fibers of $\hat\pi: E\to N$ degenerate at a uniform rate as $O((T-t)^{1/3})$.

\begin{lemma}
The map $\pi_Z : (Z,d_Z) \to (Y,d_T)$ is a homeomorphism.

\end{lemma}
Note that the target space is equipped with the metric $d_T$, not the metric $d_{\omega_Y}$. 
\begin{proof}
From Lemma \ref{lemma 1.10}, we get for any $z_1, z_2\in Z$
\begin{equation*}
d_T(\pi_Z(z_1), \pi_Z(z_2) )\le C d_{\omega_Y} ( \pi_Z(z_1),\pi_Z(z_2)  )^{\delta_0}\le C d_Z(z_1,z_2)^{\delta_0},
\end{equation*}
so the map $\pi_Z: (Z,d_Z) \to (Y,d_T)$ is continuous.

\medskip

\noindent$\bullet$ {\bf $\pi_Z$ is injective.} Suppose $z_1,z_2\in Z$ satisfies $\pi_Z(z_1) = \pi_Z(z_2) = y\in Y$. If $y\in Y^\circ$, then $z_1,z_2\in Z^\circ$, $z_1 = z_2$ by  the injectivity of $\pi_Z|_{Z^\circ}$. So we only need to consider the case $y\in Y\backslash Y^\circ = N$ and thus $z_1, z_2 \in Z\backslash Z^\circ$. Pick sequences of points $x_{1,i}, x_{2,i}\in (X,\omega_{t_i})$ converging in GH sense to $z_1, z_2$, respectively. By similar arguments as in the proof of Lemma \ref{lemma 1.3}, without loss of generality we can assume $x_{1,i}, x_{2,i}\in \hat N\subset E$. Denote $y_{1,i} = \pi(x_{1,i})$ and $y_{2,i} = \pi(x_{2,i})$. We then have
\begin{equation*}
d_{\omega_Y} ( y_{1,i}, y  ) = d_{\omega_Y}( \pi_i(x_{1,i}), \pi_Z(z_1)  )\xa{i\to\infty} 0,
\end{equation*}
and similarly $d_{\omega_Y} (y_{2,i}, y  )\to 0$ as well, and this implies that $d_{\omega_Y}(y_{1,i}, y_{2,i})\to 0$. Since $x_{1,i}$ and $x_{2,i}$ are both in the zero section $\hat N$, from the local expressions \eqref{eqn:pullback} and \eqref{eqn:pullback 1} of the metric $\omega_X = \pi^*\omega_Y + \varepsilon_0\ddbar \log \sigma_X$, we  see that $d_{\omega_X} (x_{1,i}, x_{2,i})\to 0$ as $i\to\infty$. Then by Lemma \ref{lemma 1.6} again, we get $d_{\omega_{t_i}}(x_{1,i}, x_{2,i})\le C d_{\omega_X} (x_{1,i}, x_{2,i})^\delta \to 0$. Letting $i\to\infty$ we get $d_Z(z_1, z_2) = 0$, thus $z_1 = z_2$. This proves the injectivity of $\pi_Z$. 

\medskip

\noindent $\bullet$ {\bf $\pi_Z$ is surjective.} This follows from the definition. In fact, we only need to show any $p \in Y\backslash Y^\circ = N$ lies in the image of $\pi_Z$. We fix the point $\hat p\in \hat N$ with $\hat \pi(\hat p) = p$. $\hat p_i = \hat p\in (X,\omega_{t_i})$ converge  up to subsequence in GH sense to a point $p_Z\in Z$. Then $d_{\omega_Y} (p, \pi_Z(p_Z)) = d_{\omega_Y} ( \pi_i (\hat p_i), \pi_Z(p_Z)  ) \to 0$ by definition of $\pi_i\xa{GH} \pi_Z$. It then follows that $\pi_Z(p_Z) = p$.

\medskip

Thus, $\pi_Z: (Z,d_Z) \to (Y, d_T)$ is bijective and continuous. It is also a homeomorphism since $(Z,d_Z)$ is compact.

\end{proof}

\bigskip

\noindent{\bf Acknowledgements} The author would like to thank Professors Duong H. Phong and Jian Song for their constant support, encouragement and inspiring discussions. He also wants to thank Xiangwen Zhang and Teng Fei for many helpful suggestions.  This work is  supported in part by National Science Foundation grant DMS-1710500.

\end{document}